\documentclass[a4paper,abstracton]{scrartcl}
\usepackage[utf8]{inputenc}
\usepackage{amsfonts,amsthm,amssymb,amsmath}
\usepackage[inline]{enumitem}
\usepackage{color}
\usepackage{graphicx} 
\usepackage{authblk}

\newtheorem{thm}{Theorem}
\newtheorem{lem}[thm]{Lemma}
\newtheorem{cor}[thm]{Corollary}
\newtheorem{con}[thm]{Conjecture}

\theoremstyle{remark}
\newtheorem{rmk}[thm]{Remark}

\DeclareMathOperator{\Aut}{Aut}
\DeclareMathOperator{\id}{id}
\DeclareMathOperator{\domain}{dom}
\DeclareMathOperator{\symptoms}{symp}
\DeclareMathOperator{\sync}{\stackrel{sync}{\sim}}

\title{Distinguishing infinite graphs with bounded degrees }
\author[1]{Florian Lehner}
\author[2]{Monika Pil{\'s}niak} 
\author[2]{Marcin Stawiski}
\affil[1]{Mathematics Institute, University of Warwick, \protect\\ Coventry, CV4 7AL, UK}
\affil[2]{AGH University, Department of Discrete Mathematics, \protect\\al. Mickiewicza 30, 30-059 Krakow, Poland}

\begin{document}
\maketitle
\begin{abstract}
Call a colouring of a graph distinguishing, if the only colour preserving automorphism is the identity. A conjecture of Tucker states that if every automorphism of a graph $G$ moves infinitely many vertices, then there is a distinguishing $2$-colouring. We confirm this conjecture for graphs with maximum degree $\Delta \leq 5$. Furthermore, using similar techniques we show that if an infinite graph has maximum degree $\Delta \geq 3$, then it admits a distinguishing colouring with $\Delta - 1$ colours. This bound is sharp.
\end{abstract}

\section{Introduction}
 A distinguishing $k$-colouring of a graph $G = (V,E)$ is a map $c\colon V \to \{0,\dots, k-1\}$ such that the identity is the only automorphism $\gamma$ with $c \circ \gamma = c$. The distinguishing number $D(G)$ is the least $k$ such that $G$ admits a distinguishing $k$-colouring. This notion was first introduced by Albertson and Collins \cite{alco-96} and has since received considerable attention. 
 
In this paper, we investigate connections between bounds on the maximum degree of a graph and its distinguishing number. This connection was first studied in \cite{collins-trenk-distinguishingchromatic} and \cite{klavzar-wong-zhu-distinguishinggroupaction}, where it was shown that connected finite graphs with maximum degree $\Delta$ satisfy $D(G) \leq \Delta + 1$. Equality holds if and only if $G$ is $C_5$, or $K_r$ or $K_{r,r}$ for some $r \geq 1$. By \cite{imrich}, this remains true for infinite graphs, that is, $D(G) \leq \Delta$ for every connected infinite graph.  Recently, H\"uning et al.~\cite{draft} determined the distinguishing numbers of all (finite and infinite) connected graphs of maximum degree $\Delta \leq 3$.
 
Our main result is motivated by the following conjecture by Tucker \cite{tucker} which arguably is one of the most intriguing open questions related to distinguishing numbers of infinite graphs. 
\begin{con}
\label{con:infinitemotion}
If every non-trivial automorphism of a locally finite graph $G$ moves infinitely many vertices, then $G$ admits a distinguishing $2$-colouring.
\end{con}
While this conjecture is known to be true for various graph classes (see for example \cite{draft,imrich, smith, lehner-random, lehner-intermediate,lehner-edge, Watkins}), it is still wide open in general. By results from \cite{draft}, the conjecture is true for graphs with maximum degree $\Delta \leq 3$. Our main theorem is a further step towards confirming Tucker's conjecture for graphs with bounded degrees.

\begin{thm}
\label{thm:five}
Let $G$ be a connected graph with maximum degree $\Delta \leq 5$ in which every non-trivial automorphism moves infinitely many vertices, then $D(G) \leq 2$.
\end{thm}

It is worth noting that most of the proof techniques generalise well to higher degrees. In fact, the only part of the proof that seems to break down for $\Delta > 5$ is Lemma \ref{lem:synchronise}. It is thus conceivable that a suitable generalisation of this lemma would lead to a further improvement of our bound.

Using similar proof techniques, we are also able to prove the following general bound for the distinguishing number of infinite graphs with bounded maximum degree, improving on the bound from \cite{imrich} mentioned above. It is easy to see that our bound is tight for every $\Delta \geq 3$.

\begin{thm}
\label{thm:main}
Let $G$ be an infinite connected graph with maximum degree $\Delta \geq 3$, then $D(G) \leq \Delta -1$.
\end{thm}

\section{Preliminaries}
All graphs in this paper are simple, connected and locally finite. A multirooted graph $(G,R)$ is a graph $G$ together with a set $R \subseteq V$ of roots. Note that every graph is a multirooted graph with $R = \emptyset$. An automorphism of a multirooted graph $(G,R)$ is an automorphism of $G$ which fixes $R$ pointwise. The group of automorphisms of $(G,R)$ is denoted by $\Aut (G,R)$.

A partial colouring of a graph $G$ is a function $c$ from $S \subseteq V$ to a set $C$ of colours. We denote the domain $S$ of a partial colouring by $\domain(c)$. Call two partial colourings $c$ and $c'$ \emph{compatible}, if they coincide on $\domain (c) \cap \domain (c')$. For compatible colourings $c$ and $c'$ we define the partial colouring $c \cup c'$ with domain $\domain(c) \cup \domain(c')$ by
\[
	(c \cup c')(x) := 
	\begin{cases}
	c(x) & x \in \domain(c), \\
	c'(x) & x \in \domain(c').
	\end{cases}
\]
Note that this is well defined since we required the colourings to be compatible.

We say that a partial colouring $c'$ \emph{extends} a partial colouring $c$, if $\domain (c) \subseteq \domain (c')$ and $c'$ coincides with $c$ on $\domain (c)$. Note that in this case $c \cup c' = c'$. Call a sequence $(c_i)_{i \in \mathbb N}$ of partial colourings \emph{increasing}, if $c_{i+1}$ extends $c_i$ for every $i \in \mathbb N$. For an increasing sequence of colourings, define the limit colouring by
\[
\lim_{i \to \infty} c_i := \bigcup _{i \in \mathbb N} c_i,
\]
i.e.\, the colouring that maps every $v \in \domain (c_i)$ to $c_i(v)$.

We say that an automorphism $\gamma\in \Aut(G)$ \emph{preserves} a partial colouring $c$ if $c(v) = c(\gamma v)$ whenever both colours are defined, and call the set of all $c$-preserving automorphisms the \emph{stabiliser} of $c$. Call a partial colouring of a multirooted graph \emph{$S$-distinguishing}, if every $c$-preserving automorphism must fix $S$ pointwise. Call it \emph{$S$-preserving}, if every $c$-preserving automorphism must fix $S$ setwise. As a special case of this, we call a $\domain(c)$-distinguishing colouring \emph{domain distinguishing}, and a $\domain(c)$-preserving colouring \emph{domain preserving}. Finally, a $V$-distinguishing colouring is simply called \emph{distinguishing}. The \emph{distinguishing number} of $G$ is the least number of colours in a distinguishing colouring and denoted by $D(G)$.

The following two lemmas show how $S$-distinguishing colourings can be used to construct distinguishing limit colourings. 

\begin{lem}
\label{lem:extendcolouring}
Let $(G,R)$ be a multirooted graph. Assume that we have compatible colourings $c$ and $c'$, where $c$ is an $S$-distinguishing colouring of $(G,R)$, and $c'$ is an $S'$-distinguishing colouring of $(G,R \cup S)$. Then $c \cup c'$ is an $(S \cup S')$-distinguishing colouring of $(G,R)$.
\end{lem}

\begin{proof}
Let $\gamma$ be a $(c \cup c')$-preserving automorphism. If $\gamma$ moves a vertex in $S$, then this contradicts the fact that $c$ was $S$-distinguishing. Hence $\gamma$ fixes $R \cup S$ which means that $\gamma$ is an automorphism of the multirooted graph $(G, R \cup S)$. Since $c'$ is $S'$ distingiushing for $(G,R \cup S)$ this implies that $\gamma$ cannot move any vertex in $S'$, and hence $c \cup c'$ is $(S \cup S')$-distinguishing for $(G,R)$.
\end{proof}

\begin{lem}
\label{lem:limitcolouring}
Let $(G,R)$ be a multirooted graph. Let $(c_i)_{i \in \mathbb N}$ be an increasing sequence of partial colourings. Assume that $c_i$ is $S_i$-distinguishing in $(G,R)$ where the sets $S_i$ satisfy $\bigcup_{i \in \mathbb N}S_i = V$. Then $\lim _{i \to \infty} c_i$ is distinguishing.
\end{lem}

\begin{proof}
Let $\gamma$ be an automorphism that preserves $\lim _{i \to \infty} c_i$. Then $\gamma$ preserves all partial colourings $c_i$. This implies that $\gamma$ fixes $S_i$ pointwise for every $i$, and since $V=\bigcup_{i \in \mathbb N}S_i$ we conclude that $\gamma = \id$.
\end{proof}

We say that a graph has \emph{infinite motion}, if every nontrivial automorphism moves infinitely many vertices. Conjecture \ref{con:infinitemotion} states that a locally finite graph with infinite motion has distinguishing number at most $2$. Results in \cite{imrich, Watkins} imply the following theorem which says that the conjecture is true for trees.

\begin{thm}
\label{thm:treelike}
Let $G$ be a locally finite tree with infinite motion, then $D(G)\leq 2$. In particular the distinguishing number of a locally finite tree without leaves is at most $2$.
\end{thm}

\section{Proof of the main result}

Before we prove our main theorem, we introduce a few concepts used in the proof and prove some basic facts about them.

Let $c$ be a domain preserving colouring of a multirooted graph $(G,R)$. Note that the stabiliser of a domain preserving colouring forms a group (this is not true in general for the stabiliser of a partial colouring). Call a vertex $v \in V \setminus R$ \emph{charted}, if it is coloured or it has a neighbour in $R$, and call $v$ \emph{uncharted otherwise}. A \emph{moving tuple (singleton, pair, triple, quadruple,\dots)} is an orbit of a charted vertex under the stabiliser of $c$. Since $c$ is domain preserving, all vertices in a moving tuple must be charted.

Note that all members of a moving tuple must have the same neighbours in $R$. Outside of $R$ there may be vertices that are adjacent to some, but not all members of a moving tuple. We call such vertices \emph{uncommon neighbours} of the tuple. Uncommon neighbours can synchronise the action on different moving tuples as the following lemma demonstrates.

\begin{lem}
\label{lem:synchronise}
Let $c$ be a domain preserving colouring, and let $A$ and $B$ be moving tuples with at most $3$ elements. If $A$ has an uncommon neighbour in $B$, then $B$ has an uncommon neighbour in $A$ and there is a bijection $f\colon A \to B$ such for every automorphism $\gamma$ in the stabiliser of $S$ we have $\gamma \mid_B = f \circ \gamma \mid _A \circ f^{-1}$. In particular, every such $\gamma$ that fixes $v \in A$ must also fix $f(v) \in B$ and vice versa.
\end{lem}

\begin{proof}
First note that by the definition of moving tuples the stabiliser of $S$ fixes $A$ and $B$ setwise but acts transitively on the elements of $A$ and $B$ respectively. Hence every element of $A$ has the same number of neighbours in $B$ and vice versa. Since $A$ and $B$ both contain at most $3$ vertices, this can only work if the graph induced by the edges between $A$ and $B$ is either complete bipartite, or a matching, or a $6$-cycle. In the first case $B$ cannot contain an uncommon neighbour of $A$, in the second case we take $f$ to be the map that takes each vertex to the other endpoint of its matching edge, in the last case let $f$ be the map taking each vertex to the antipodal vertex on the cycle.
\end{proof}

\begin{rmk}
Note that the above lemma is no longer true if we allow moving quadruples. In this case it is possible that we have two disjoint copies of $K_{2,2}$ between two moving quadruples or two disjoint copies of $K_{2,1}$ between a moving pair and a moving quadruple. This is also why our proof does not easily carry over to higher maximum degrees.
\end{rmk}

We call two moving pairs or triples $A$ and $B$ synchronised and write $A \sync B$, if there is a finite sequence $A = A_1, A_2, \dots, A_k = B$ of moving pairs or triples such that $A_i$ has a uncommon neighbour in $A_{i+1}$ for $1 \leq i \leq k-1$. By the above lemma, being synchronised is an equivalence relation. Furthermore, between any two synchronised moving pairs or triples we can find a bijection $f$ such that an automorphism in the stabiliser of $S$ fixes $x \in A$ if and only if it fixes $f(x) \in B$.

Let $c$ be a partial colouring of a multirooted graph $(G,R)$ and let $v$ be a vertex in the domain of $c$. The \emph{monochromatic component} of $v$ is the set of all vertices that can be reached from $v$ by a monochromatic path (i.e.\ no vertices of the opposite colour, but also no uncoloured vertices). We say that a monochromatic component $K$ is \emph{healthy}, if one of the following holds:
\begin{itemize}
\item the vertices in $K$ do not have colour $0$,
\item $K \cap R \neq \emptyset$,
\item $|K| < \infty$ and $K$ has no uncoloured neighbours, or
\item there is a vertex in $K$ with at least $4$ neighbours in $K$.
\end{itemize}
If a monochromatic component is not healthy, we call it \emph{unhealthy}. Call the colouring $c$ \emph{healthy}, if all monochromatic components under $c$ are healthy and \emph{unhealthy} otherwise. A \emph{symptom} of a colouring $c$ is a vertex in an unhealthy component with uncoloured neighbours outside of $R$, or an unhealthy infinite monochromatic component. Denote by $\symptoms (c)$ the set of symptoms of a colouring $c$. Note that a  colouring $c$ is healthy if and only if it has no symptoms.

\begin{lem}
\label{lem:healthyextend}
Let $c$ be a healthy colouring and let $c'$ be a colouring extending $c$ such that $\domain (c') \setminus \domain (c)$ is finite and contains no symptoms of $c'$. Then $c'$ is healthy. 
\end{lem}

\begin{proof}
If $c'$ is unhealthy, then there is a symptom $A$ of $c'$. By assumption, $A \cap \domain(c) \neq \emptyset$. If $A$ is a vertex with uncoloured neighbours, then $A$ is a symptom of $c$. Otherwise there is an infinite monochromatic component of $c$ contained in $A$, which is a symptom of $c$. In both cases we obtain a contradiction to the assumption that $c$ is healthy.
\end{proof}

\begin{lem}
\label{lem:healthylimit}
The limit colouring of any increasing sequence $(c_i)_{i \in \mathbb N}$ of healthy partial colourings is healthy.
\end{lem}

\begin{proof}
Assume that there is an unhealthy component $K$ in the limit colouring. Let $i \in \mathbb N$ such that $K_i := K \cap \domain(c_i) \neq \emptyset$. Since $K$ is unhealthy in $c$, all vertices of $K_i$ are coloured with colour $0$, $K_i \cap R = \emptyset$, and no vertex of $K_i$ has $4$ or more neighbours in $K_i$. The colouring $c_i$ is healthy, so $K_i$ is finite and all neighbours of $K_i$ except those in $R$ have already been coloured (with colours different from $0$). It follows that $K = K_i$, and thus $K$ is healthy, a contradiction.
\end{proof}

We are now ready to prove Theorem \ref{thm:five}. Before doing so, let us briefly recall its statement.

\begingroup
\def\thethm{\ref{thm:five}}
\begin{thm}
Let $G$ be a connected graph with maximum degree $\Delta \leq 5$ in which every non-trivial automorphism moves infinitely many vertices, then $D(G) \leq 2$.
\end{thm}
\addtocounter{thm}{-1}
\endgroup

\begin{proof} [Proof of Theorem \ref{thm:five}]
If $G$ is a tree, then we are done by Theorem \ref{thm:treelike}. We may thus assume that $G$ contains at least one cycle. Let $C$ be an induced cycle, let $P$ a geodesic ray which meets $C$  only at its starting point, and let $s$ be a neighbour of the starting point of $P$ on $C$. Define $R = P+C-s$.

We will now inductively define sets $S_i$ with $\bigcup_{i \in \mathbb N} S_i = V$ and an increasing sequence $c_i$ of partial colourings such that for every $i \in \mathbb N$
\begin{enumerate}[label = (C\arabic*)]
\item \label{itm:sidistinguishing}
$c_i$ is $S_i$-distinguishing in $(G,R)$,
\item \label{itm:neighbourofroot}
$R$ is contained in a monochromatic component $K_R$ of colour $0$ such that
\begin{itemize}
\item each $r \in R$ has at most one neighbour in $K_R - R$, and
\item each $v \in K_R - R$ has exactly one neighbour $r$ in $K_R$ which lies in $R-C$, neither $v$ nor $r$ have uncoloured neighbours,
\end{itemize}
\item \label{itm:healthy}
$c_i$ is healthy.
\end{enumerate}

Before we construct these colourings, we show that their limit colouring will be distinguishing.

First note that by Lemma \ref{lem:limitcolouring} and \ref{itm:sidistinguishing}, the limit colouring $c = \lim_{i \to \infty} c_i$ will be distinguishing for the multirooted graph $(G,R)$. So we only need to show that every colour preserving automorphism fixes $R$ pointwise. 

It is easy to see that property \ref{itm:neighbourofroot} carries over to the limit colouring. Hence the monochromatic component $K_R$ with respect to $c$ is a ray with leaves attached to it. There is no leaf attached to the first two elements of the ray (because they are in $C$). Hence any automorphism which fixes $K_R$ setwise must fix $R$ pointwise, and we only need to show that any colour preserving automorphism has to fix $K_R$ setwise. 

To this end, observe that by Lemma \ref{lem:healthylimit} and \ref{itm:healthy}, the limit colouring is healthy. In particular, $K_R$ is the only monochromatic component of colour $0$ which does not contain a vertex of degree at least $4$. Hence $K_R$ must be fixed setwise by every colour preserving automorphism.

In order to recursively construct the colourings $c_i$ we make a few more assertions which will be true for every $i \in \mathbb N$

\begin{enumerate}[resume,label = (C\arabic*)]
\item \label{itm:fixcoloured}
$c_i$ is domain preserving in $(G,R)$,
\item \label{itm:domfinite}
$\domain (c_i) \setminus R$ is finite, 
\item \label{itm:domneighbours}
$\domain(c_i)$ contains all neighbours of $S_i$,
\item \label{itm:noktuples}
there are no moving $k$-tuples for $k \geq 4$ in $\domain(c_i)$.
\end{enumerate}

Let $S_1 = R$, and let $c_1$ be the colouring that assigns $0$ to all vertices in $R$ and leaves the remaining vertices uncoloured. Properties \ref{itm:sidistinguishing} to \ref{itm:domneighbours} are trivially satisfied. To see that \ref{itm:noktuples} also holds, note that the vertices of $R$ form a ray. Every internal vertex of this ray has two neighbours in $R$, and thus at most $3$ neighbours outside of $R$. The starting point of $R$ may have $4$ neighbours outside of $R$, but at least one of them is also a neighbour of an internal vertex of $R$ and hence contained in a moving $k$-tuple for $k \leq 3$.

Now assume that we already defined $S_i$ and $c_i$. Let $x$ be a vertex at minimal distance to $C$ which is not contained in $S_i$. Note that $x$ is either a neighbour of $S_i$ or a neighbour of $R$. Since by \ref{itm:domneighbours} all neighbours of $S_i$ are coloured, we conclude that $x$ is charted. Let $X$ be the moving tuple containing $x$.

We claim that if $X$ is not a singleton, then one of the following holds:
\begin{enumerate}
\item There is an uncoloured  $Y \sync X$ whose unique coloured neighbour lies in $R - C$.
\item There is $Y \sync X$ which has uncharted uncommon neighbours.
\end{enumerate}

 Let $\mathcal Y$ be the set of tuples that are synchronised with $X$. If no tuple in $\mathcal Y$ has an uncharted uncommon neighbour, then there is an automorphism that moves only vertices contained in $\mathcal Y$ and fixes all other vertices. Since $G$ has infinite motion, this implies that $\mathcal Y$ must be infinite. By \ref{itm:domfinite}, $\domain(c_i)\setminus R$ is finite,  so infinitely many elements of $\mathcal Y$ must be neighbours of $R$.
 
Note that every internal vertex of $R$ has neighbours in at most one moving tuple in $\mathcal Y$, otherwise its degree would be larger than $5$ since it already has two neighbours in $R$. It's easy to see that there are  tuples $Y_1, Y_2, Y_3\in \mathcal Y$ whose only coloured neighbours lie on the geodesic ray $P \subseteq R$  such that $Y_i$ has uncommon neighbours in $Y_{i+1}$ for $i \in \{1,2\}$. Since each $Y_i$ has $2$ or more neighbours in $R$, we conclude that at least two of those neighbours, $v$ and $w$, lie at distance $5$ or more from each other on $R$. 

Assume that $v$ is a neighbour of $Y_1$ and $w$ is a neighbour of $Y_3$. Then there is a path $v,y_1,y_2,y_3,w$, where $y_{i+1} \in Y_{i+1}$ is an uncommon neighbour of $y_i \in Y_i$ for $i \in \{1,2\}$. This path has length $4$, contradicting the fact that $P$ was geodesic. The cases where $v$ and $w$ are neighbours of other tuples in $Y_1,Y_2,Y_3$ are completely analogous and lead to even shorter paths. This finishes the proof of the claim.

Let $S_{i+1} = S_i \cup \{x\}$. We now define the colouring $c_{i+1}$ by first colouring $\domain(c_i)$ according to $c_i$, and then possibly colouring some yet uncoloured vertices. If $X$ is uncoloured (this can only happen if $x$ is a neighbour of $R$), then colour all elements of $X$ with $1$. Note that if $X = \{x\}$ is a moving singleton with no uncoloured neighbours, then \ref{itm:sidistinguishing} to \ref{itm:fixcoloured} are trivially satisfied, whence we found $c_{i+1}$.  

Otherwise distinguish cases according to the two possible options arising from the above claim. The case where $X = \{x\}$ is a singleton with uncoloured neighbours will be treated together with the case where there are uncharted uncommon neighbours of some $Y \sync X$.

First assume that there is an uncoloured $Y \sync X$ whose unique coloured neighbour $r$ lies in $R-C$. By Lemma \ref{lem:synchronise} there is an element $y \in Y$ such that any $c_i$-preserving automorphism that fixes $y$ must also fix $x$. Colour $y$ with $0$ and all other uncoloured neighbours of $r$ with $1$. Colour all uncoloured neighbours of $x$ and $y$ with colour $1$ and let $c_{i+1}$ be the resulting colouring. 

It remains to show that $c_{i+1}$ satisfies \ref{itm:sidistinguishing} to \ref{itm:noktuples}. For \ref{itm:sidistinguishing} note that any $c_{i+1}$-preserving automorphism also preserves $c_i$ and thus fixes $S_i$ pointwise and $Y$ setwise. Since $y$ is the only vertex with colour $0$ in $Y$, every $c_{i+1}$-preserving automorphism must fix $y$ and thus also $x$. If $r$ had any neighbours of colour $0$ in $c_i$, then $c_i$ would violate \ref{itm:neighbourofroot}. Hence the way we picked $y$ and the fact that we coloured all remaining neighbours of $y$ and $r$ with $1$ ensures that \ref{itm:neighbourofroot} holds for $c_{i+1}$. Property \ref{itm:healthy} holds by Lemma \ref{lem:healthyextend} and the fact that $y$ is the only vertex with colour $0$ in $\domain (c_{i+1}) \setminus \domain(c_i)$. For \ref{itm:fixcoloured}, note that $c_i$ was domain preserving, hence every vertex in $\domain(c_i)$ is mapped to a coloured vertex. Furthermore, every $c_{i+1}$-preserving automorphism must fix $x$, $y$, and $r$ and hence fixes their neighbourhoods setwise. Properties \ref{itm:domfinite} and \ref{itm:domneighbours} hold because we coloured finitely many vertices including the neighbourhood of $x$. For the proof of \ref{itm:noktuples} first note that when we picked $Y$, the tuple $X$ was already coloured. So $X \neq Y$ and consequently both $x$ and $y$ have at least two $c_i$-charted neighbours: one in a synchronised moving tuple, and one in $S_i$ or $R$ respectively. All other vertices in $\domain(c_{i+1}) \setminus \domain(c_i)$ are $c_i$ charted. Since $c_i$ satisfies \ref{itm:noktuples}, all charted vertices are contained in moving $k$-tuples for $k \leq 3$ and we conclude that $c_{i+1}$ satisfies \ref{itm:noktuples} as well.

Finally, consider the case when there is $Y \sync X$ which has uncharted uncommon neighbours. If $X$ has uncharted uncommon neighbours, we pick $Y = X$. Let $y \in Y$ be such that every $c_i$-preserving automorphism that fixes $y$ must also fix $x$. The argument below also applies in the case where $X$ is a singleton with uncoloured neighbours, in this case let $Y = X$ and $y = x$. 

If $X \neq Y$ and $X$ has uncoloured neighbours, then colour all of them with colour $1$. If $Y$ is uncoloured, then colour it with colour $1$. Let $\mathcal A = Y$ and iteratively run the following procedure.

Let $a \in \mathcal A$ (in the first step choose $a = y$) and let $A$ be the moving tuple containing $a$. Colour all neighbours of $A$ that are also neighbours of $R$ with colour $1$. If there are still uncoloured neighbours, we distinguish the following cases.
\begin{enumerate}
\item \label{itm:common} $A$ has no uncoloured uncommon neighbours (this includes the case where $|A| = 1$).
\begin{enumerate}[label=\arabic{enumi}\Alph*]
\item \label{itm:onetwothreecommon} If $A$ has $3$ or fewer uncoloured neighbours, do nothing.
\item \label{itm:fourcommon} If $A$ has $4$ uncoloured neighbours, then colour $3$ of them with $0$.
\end{enumerate}
\item \label{itm:uncommon} $A = \{a_1,a_2\}$ is a moving pair, or $A = \{a_1,a_2,a_3\}$ is a moving triple with uncoloured uncommon neighbours. Without loss of generality, $a_1 = a$. We say that an uncommon neighbour $v$ of $A$ is \emph{linked to $a_i$} if either $a_i$ is the unique neighbour of $v$ in $A$, or if $A$ is a moving triple and $a_i$ is the unique vertex in $A$ that is not connected to $v$.
\begin{enumerate}[label=\arabic{enumi}\Alph*]
\item \label{itm:onelinked} If there is a unique uncommon neighbour linked to $a_1$, then colour it with $0$.
\item \label{itm:twothreelinked} If there are $2$ or $3$ uncommon neighbours linked to $a_i$, then colour $(i-1)$ uncommon neighbours linked to $a_i$ with colour $0$.
\item \label{itm:fourlinked} If there are $4$ uncommon neighbours linked to $a_i$, then colour $(4-i)$ uncommon neighbours linked to $a_i$ with colour $0$.
\end{enumerate}
\end{enumerate}
Finally colour all remaining uncoloured neighbours of $A$ with colour $1$. If any of the newly coloured vertices are symptoms, then add them to $\mathcal A$. Remove all elements from $\mathcal A$ that have ceased to be symptoms. If $\mathcal A \neq \emptyset$, iterate.

We claim that after finitely many iterations this yields the desired colouring $c_{i+1}$. We first show that \ref{itm:sidistinguishing}, \ref{itm:neighbourofroot}, and \ref{itm:fixcoloured} to \ref{itm:noktuples} hold after every step of the recursion. Then we prove that the recursion eventually terminates with $\mathcal A = \emptyset$ and that this implies that the colouring is healthy.
 
We show inductively that all desired properties except \ref{itm:healthy} hold after every iteration. The induction basis for \ref{itm:neighbourofroot}, \ref{itm:fixcoloured}, and \ref{itm:domfinite} trivially follows from the corresponding properties of $c_i$.

For the induction basis for \ref{itm:sidistinguishing} first note that if $|X| = 1$, then this is trivially satisfied (even before the first iteration). To see that if $|X|>1$ it is satisfied after the first iteration, note that $Y$ has uncharted uncommon neighbours. Hence it still has uncoloured uncommon neighbours after colouring the neighbours of $R$ in its neighbourhood, i.e.\ we apply Case \ref{itm:uncommon}. In each of the subcases it is easy to see that the number of uncommon neighbours of $y = a_1$ with colour $0$ differs from the corresponding numbers for $a_2$ and $a_3$, whence $y$ (and thus also $x$) is fixed by every colour preserving automorphism. 

If $Y = X$, then property \ref{itm:domneighbours} will be satisfied after the first iteration where all neighbours of $Y$ are coloured, otherwise it follows from the colouring that was done before the first iteration. 

Property \ref{itm:noktuples} always holds before the first iteration. If $X=Y$, then this follows from \ref{itm:noktuples} for $c_i$. Otherwise $X$ has no uncharted uncommon neighbours, so all uncharted neighbours of $X$ are neighbours of $x$. Note that $x$ has at least one neighbour in $S_i$ and one neighbour in a synchronised moving tuple, and thus $X$ has at most $3$ uncharted neighbours. Since $c_i$ satisfies \ref{itm:noktuples}, every charted neighbour of $X$ is contained in moving $k$-tuples for some $k \leq 3$. Hence there are no moving $k$-tuples for $k \geq 4$ in the neighbourhood of $X$ and it follows that the colouring before the first iteration satisfies~\ref{itm:noktuples}.

For the induction step note that the colouring procedure preserves properties \ref{itm:sidistinguishing} since an extension of an $S$-distinguishing colouring is again $S$-distinguishing. It preserves \ref{itm:neighbourofroot} since throughout the procedure neighbours of root vertices are only coloured with colour $1$. It preserves \ref{itm:fixcoloured}, since the whole neighbourhood of an orbit is coloured in each recursion step. It preserves \ref{itm:domfinite}, since only finitely many vertices are coloured in each iteration. It trivially preserves \ref{itm:domneighbours}.

For the induction step for \ref{itm:noktuples}, note that in all cases except \ref{itm:onelinked} with a moving triple, every element of $A$ has a different number of neighbours with colour $0$. In particular, $A$ must be fixed pointwise by every colour preserving automorphism and the colouring on the neighbourhood of $A$ makes sure that there are no moving $k$-tuples for $k \geq 4$. If $A$ is a moving triple and we are in case \ref{itm:onelinked}, then $a_1$ is fixed by every colour preserving automorphism, and so is the unique previously uncoloured neighbour that is linked to $a_1$. The vertices $a_2$ and $a_3$ are either fixed by every colour preserving automorphism or they form a moving pair, and the same is true for the previously uncoloured vertices that are linked to them. Every remaining uncoloured neighbour of $A$ is not uncommon and hence connected to ervery vertex of $A$. Since $a_1$ had at least one coloured vertex before the iteration (otherwise it would not have been added to $\mathcal A$) we conclude that there are at most $3$ uncoloured neighbours of $\mathcal A$ left, and hence there is no moving $k$-tuple for $k \geq 4$.

Next we show that the recursion terminates. For this purpose, first note that after each iteration all neighbours of every vertex in $A$ are coloured. Hence the members of $A$ are removed from $\mathcal A$ and won't be added again in later iterations. Hence it suffices to show that the number of vertices that can possibly be added to $\mathcal A$ is finite.

To this end define the generation  $g(x)$ of an element $x \in \mathcal A$ as follows. The generation of any element of $Y$ is $1$. If $x$ has been added to $\mathcal A$ in an iteration where $g(a)=i$, then $g(x) = i+1$. Note that every vertex in generation $i+1$ has at least one neighbour in generation $i$ and thus lies at distance at most $i-1$ from $Y$. So it suffices to show that there is an upper bound on $g(x)$.

Every vertex added to $\mathcal A$ is a symptom and thus has colour $0$. Hence every vertex of generation $3$ and higher has a neighbour of colour $0$ of an earlier generation. Note that moving triples of colour $0$ can only be created in Subcases \ref{itm:fourcommon} and \ref{itm:fourlinked}. In case $a$ already has a neighbour of colour $0$ these triples will be contained in a healthy component and hence not be added to $\mathcal A$.  In particular, no $x \in \mathcal A$ with $g(x) \geq 4$ is contained in a moving triple.

Next note that if $A$ is a moving pair, then $A$ is fixed pointwise by every automorphism that preserves the colouring after the iteration. Hence the vertices of any moving tuple that is added to $\mathcal A$ this iteration must have a common neighbour in $A$. In particular, the members of a moving tuple of generation $5$ or higher lie in the same monochromatic component. If Subcase \ref{itm:fourlinked} is applied to such a tuple, no vertices are added to $\mathcal A$. Since all other cases only yield moving singletons of colour $0$, we conclude that every $x \in \mathcal A$ with $g(x) \geq 6$ is fixed pointwise by any colour preserving automorphism. Neither of the subcases of Case \ref{itm:common} yields new elements of $\mathcal A$ after generation $4$ and it follows that there are no vertices of generation $7$. 

Finally we need to show that the colouring $c_{i+1}$ obtained by running the recursion until $\mathcal A = \emptyset$ is healthy. Assume that there is a symptom $v \in \domain(c_{i+1}) \setminus \domain (c_i)$. Then $v$ would have been a symptom after the iteration in which it was coloured, and hence we would have added $v$ to $\mathcal A$. Since we ran the recursion until $\mathcal A = \emptyset$, there must have been an iteration in which $v$ was removed from $\mathcal A$ due to not being a symptom anymore. This implies that $v$ is not a symptom of any subsequent colouring and in particular $v$ is not a symptom of $c_{i+1}$. Since there are no symptoms in $\domain(c_{i+1}) \setminus \domain (c_i)$, we can use Lemma \ref{lem:healthyextend} and conclude that $c_{i+1}$ is healthy.
\end{proof}

It is obvious that by using more colours in the above colouring procedure we can avoid moving $k$-tuples for $k \geq 4$ even if the maximum degree is larger than $5$. Consequently, the proof can be adapted to show the following.

\begin{cor}
Let $G$ be a graph with maximum degree $\Delta \geq 3$ in which every non-trivial automorphism moves infinitely many vertices. Then $D(G) \leq \frac \Delta 3 + 1$.
\end{cor}

\section{A general bound}

In light of the above corollary it is only natural to ask how much the bound for $D(G)$ changes if we allow automorphisms moving only finitely many vertices. Recall that by \cite{imrich} we know that $D(G) \leq \Delta$ for any infinite graph with maximum degree $\Delta$. On the other hand, it is not hard to construct examples of infinite graphs with $D(G) = \Delta-1$. Simply attach $\Delta - 1$ new leaves to a vertex of degree $1$ in any infinite graph. Hence the bound given in Theorem \ref{thm:main} (which we recall below) is tight.

\begingroup
\def\thethm{\ref{thm:main}}
\begin{thm}
Let $G$ be an infinite connected graph with maximum degree $\Delta \geq 3$, then $D(G) \leq \Delta -1$.
\end{thm}
\addtocounter{thm}{-1}
\endgroup

\begin{proof}
If $G$ is a leafless tree, then we are done by Theorem \ref{thm:treelike}. 

Hence we may assume that $G$ has either a cycle or a vertex of degree $1$. If there is a cycle, then let $C$ be an induced cycle, let $P$ be a geodesic ray which meets $C$ only at its starting point, and let $s$ be a neighbour of the starting point of $P$ on $C$. Define $R=P+C-s$. If $G$ is a tree with leaves, then let $R$ be a ray starting at a leaf.

We will now inductively define domain distinguishing partial colourings $c_i$ of $(G,R)$ with connected domain $S_i = \domain (c_i)$. For $c_0$ colour all vertices of $R$ with colour $0$. This is trivially domain distinguishing for $(G,R)$ and clearly $S_0 = R$ is connected. 

For the recursion step, define an equivalence relation on $V \setminus S_i$ by $x \sim y$ if $x$ and $y$ have the same neighbours in $S_i$. Denote by $[x]$ the equivalence class of $x$ with respect to this relation. Let $v \notin S_i$ be an uncoloured vertex which lies as close as possible to $r$. Then $v$ has a neighbour $s$ in $S_i$. Since $S_i$ is connected, $s$ has at most $\Delta - 1$ neighbours outside of $S_i$ whence $|[v]| \leq \Delta - 1$. Colour  all vertices in $[v]$ with different colours. If $|[v]| \leq \Delta - 2$, avoid the colour $0$. The resulting colouring $c_{i+1}$ clearly is domain distinguishing for $(G,S_i)$. By Lemma \ref{lem:extendcolouring} and the induction hypothesis it is also domain distinguishing for $(G,R)$.

The sequence $c_i$ satisfies the conditions of Lemma \ref{lem:limitcolouring}, thus we get a limit colouring $c$ which is distinguishing for $(G,R)$. In order to show that $c$ is distinguishing it suffices to show that every colour preserving automorphism must fix $R$ pointwise. 

This can be achieved by showing that $R$ is the unique monochromatic ray of colour $0$ satisfying the following property: 
\begin{enumerate}[label=(\textasteriskcentered)]
\item \label{itm:Rfirstvertex} the first vertex is either a leaf, or has a common neighbour outside of $R$ with another vertex of $R$.
\end{enumerate}

Clearly $R$ satisfies \ref{itm:Rfirstvertex}. If $v$ is a neighbour of $R$, then $|[v]| \leq \Delta - 2$. This is obvious, if $v$ is a neighbour of an inner vertex of the ray induced by $R$. For the starting vertex it follows from the fact that this vertex either has degree $1$ or has a common neighbour with an inner vertex of the ray. Hence no neighbour of $R$ is coloured with colour $0$ showing that any other ray satisfying \ref{itm:Rfirstvertex} must be disjoint from $R$.

Further observe that if $u$ and $v$ are neighbours which are both coloured $0$ then they must have been coloured in different steps of the construction. If $v$ is coloured later than $u$, then $\Delta - 1$ neighbours of $u$ (including $v$) are coloured in the same step as $v$.

Now let $Q = q_0q_1q_2\dots$ be a monochromatic ray of colour $0$ disjoint from $R$. We claim that the vertices of $Q$ are coloured in order, i.e.\ $q_i$ is coloured before $q_{i+1}$. Indeed, if $q_{i+1}$ is coloured earlier than $q_i$, then let $q$ be the first vertex of $\{q_j\mid j \geq i\}$ that is coloured. Both neighbours of $q$ on $Q$ are coloured later than $q$. Hence by the above observation $q$ has at least $2(\Delta - 1)$ neighbours, a contradiction to $\Delta$ being the maximum degree. 

At least one neighbour of $q_0$ is coloured before $q_0$, so if $q_0$ is a leaf, then $q_1$ can't be coloured after $q_0$. Furthermore, once $q_1$ is coloured all neighbours of $q_0$ are coloured. But for $i>0$ the only neighbour of $q_i$ which is not coloured later than $q_i$ is $q_{i-1}$, whence $q_i$ can't have any common neighbours with $q_0$ outside of $Q$. Thus $Q$ does not satisfy~\ref{itm:Rfirstvertex}.
\end{proof}

\section*{Acknowledgements}
This work was partially supported by Ministry of Science and Higher Education of Poland and OEAD grant no.\ PL 08/2017. Florian Lehner was supported by the Austrian Science Fund (FWF), grant J 3850-N32.

\bibliographystyle{abbrv}
\bibliography{sources.bib}
\end{document}